\documentclass[11pt]{amsart}
\usepackage{amsfonts, amsthm, amssymb, amsmath, stmaryrd}
\usepackage[utf8]{inputenc}
\usepackage{mathrsfs,array}
\usepackage{eucal,color,times,enumerate,accents}
\usepackage[all]{xy}
\usepackage{url}
\usepackage{enumitem}
\usepackage[pdftex,bookmarksnumbered,bookmarksopen]{hyperref}
\hypersetup{
  colorlinks,
  citecolor=black,
  linkcolor=black,
  urlcolor=black}
  
  \newcommand\blfootnote[1]{%
  \begingroup
  \renewcommand\thefootnote{}\footnote{#1}%
  \addtocounter{footnote}{-1}%
  \endgroup
}

  \newcommand{\Addresses}{{
  \bigskip
  \footnotesize

\textsc{London School of Geometry and Number Theory, UCL, Department of Mathematics, Gower street, WC1E 6BT, London, UK}\par\nopagebreak
  \textit{E-mail address}, G.~Baldi: \texttt{gregorio.baldi.16@ucl.ac.uk}

}}

\usepackage{tikz-cd}
\usepackage{bbm}

\usetikzlibrary{positioning}
\usepackage{dsfont}
\usepackage{tikz}
\usetikzlibrary{matrix,arrows,decorations.pathmorphing,positioning}
\usepackage{calligra}
\usepackage[colorinlistoftodos]{todonotes}
\usepackage{xy}
\xyoption{all}

\DeclareMathOperator{\Ag}{\mathcal{A}_g}

\usepackage{tikz}
\usetikzlibrary{matrix,arrows,decorations.pathmorphing}

\usepackage{leftidx}

\input xy
\xyoption{all}

\usepackage[bottom=3.2cm, top=2.8cm, left=2.6cm, right=2.6cm]{geometry}

\usepackage{tikz-cd}
\theoremstyle{plain}
\newtheorem{thm}{Theorem}[section]

\newtheorem{conj}{Conjecture}[section]

\newtheorem{prop}[thm]{Proposition}
\newtheorem{cor}[thm]{Corollary}
\theoremstyle{definition}
\newtheorem{defi}[thm]{Definition}

\newtheorem*{rmk}{Remark}

\newtheorem*{exe}{Example}

\theoremstyle{remark}

\numberwithin{equation}{subsection}

\DeclareMathOperator{\End}{End}

\DeclareMathOperator{\Hom}{Hom}
\DeclareMathOperator{\Gal}{Gal}

\DeclareMathOperator{\Res}{Res}

\DeclareMathOperator{\im}{Im}

\DeclareMathOperator{\CM}{CM}
\DeclareMathOperator{\Gl}{GL}

\DeclareMathOperator{\Sl}{SL}

\DeclareMathOperator{\GSp2g}{GSp_{2g}}

\newcommand{\Gm}{\mathbb{G}_m}
\newcommand{\Ga}{\mathbb{G}_a}

\newcommand{\tors}{\operatorname{tors}}
\newcommand{\Aut}{\operatorname{Aut}}

\newcommand{\Gad}{G^{\text{ad}}}

\newcommand{\Z}{\mathbb{Z}}
\newcommand{\Q}{\mathbb{Q}}

\newcommand{\R}{\mathbb{R}}
\newcommand{\Pp}{\mathbb{P}}
\newcommand{\A}{\mathbb{A}}

\newcommand{\Oo}{\mathcal{O}}

\newcommand{\Hh}{\mathbb{H}}

\newcommand{\C}{\mathbb{C}}
\newcommand{\DT}{\mathbb{S}}

\newcommand{\Qbar}{\overline{\mathbb{Q}}}

\makeatletter
\def\subtitle#1{\gdef\@subtitle{#1}}
\def\@subtitle{}
\makeatother

\usepackage[utf8]{inputenc}
\usepackage[T1]{fontenc}

\begin{document}
\title{On a conjecture of Buium and Poonen}\blfootnote{To appear in Annales de l'Institut Fourier.}
\author{Gregorio Baldi}

\begin{abstract}
Given a correspondence between a modular curve $S$ and an elliptic curve $A$, we prove that the intersection of any finite-rank subgroup of $A$ with the set of points on $A$ corresponding to an isogeny class on $S$ is finite. The question was proposed by A. Buium and B. Poonen in 2009. We follow the strategy proposed by the authors, using a result about the equidistribution of Hecke points on Shimura varieties and Serre's open image theorem. The result is an instance of the Zilber-Pink conjecture.
\end{abstract}

\maketitle

\section{Introduction}
A. Buium and B. Poonen \cite{buiumpoonen1} studied the problem of independence of points on elliptic curves arising from special points on modular and Shimura curves. As a first approximation the problem can be described as follows. Let $S/\Qbar$ be a modular curve, $A/\Qbar$ an elliptic curve and $\Gamma_0 \leq A(\Qbar)$ a finitely generated subgroup. Let
\begin{displaymath}
\Psi : S \longrightarrow A /\Qbar
\end{displaymath}
be a (non-constant) morphism and $\CM\subset S(\Qbar)$ be the set of special points of $S$, i.e. the points corresponding to elliptic curves with complex multiplication (also referred as $\CM$-elliptic curves). One can ask whether the following are true:
\begin{enumerate}
\item Andr\'{e}-Oort-Manin-Mumford: $\Psi (\CM)\cap A_{\tors}$ is finite;
\item Andr\'{e}-Oort-Mordell-Lang: $\Psi (\CM)\cap \Gamma_0$ is finite.
\end{enumerate}
The first statement is an easy consequence of the \emph{Andr\'{e}-Oort conjecture} for $S\times A$. We recall the shape of the Andr\'{e}-Oort conjecture (AO from now on) for products of a modular curve and an elliptic curve. This is a particular case of a theorem of Pila (\cite[Theorem 1.1]{pila}, see also \cite{MR2411018}).
\begin{thm}[Andr\'{e}-Oort-Manin-Mumford]\label{aomm}
Let $S$ be a modular curve, $A$ an elliptic curve and consider their product $T:=S \times A$. A point $(s,a)\in T$ is said to be special if $s\in \CM$ and $a\in A_{\tors}$. The only irreducible closed subvarieties of $T$ containing a Zariski dense set of special points are: $\{\CM- \text{point}\}\times \{\text{torsion point}\}$, $S\times \{\text{torsion point}\}$, $\{\CM-\text{point}\}\times A$, $S\times A$.
\end{thm}
It is interesting to notice that (1), together with the modularity theorem of Wiles, Breuil, Conrad, Diamond, Taylor, implies that there are only finitely many torsion Heegner points on any elliptic curve over $\Q$ (first proven in \cite{neko}). For a complete discussion about this, we refer to \cite[Section 1.2]{buiumpoonen1} and the references therein.

Statement (2) is true because there are only finitely many classes of $\Qbar$-isomorphic $\CM$-abelian varieties of a given dimension defined over a given number field. For example, in the case of elliptic curves, it is a classical result in the theory of complex multiplication that the set of $\CM$-points of $X_1(N)$, defined over a given number field, is finite.

A. Buium and B. Poonen \cite[Theorem 1.1]{buiumpoonen1} were able to deal with finitely generated subgroups and torsion points simultaneously. The main theorem they discuss is as follows.
\begin{thm}[Buium-Poonen]\label{thmcm}
Let $A / \Qbar$ be an elliptic curve, $\Psi : X_1(N)\to A$ be a non-constant morphism defined over $\Qbar$. Let $\Gamma \leq A(\Qbar)$ be a finite rank subgroup, i.e. the division hull of a finitely generated subgroup $\Gamma_0 \leq A(\Qbar)$. Then $\Psi (\CM)\cap \Gamma$ is finite.
\end{thm}
Since the proof is very elegant and our results will follow a similar path, we present here the main points of the strategy. It relies on two deep results from equidistribution theory due to Zhang (namely \cite[Corollary 3.3]{MR2200081} and \cite[Theorem 1.1]{almostdivision}) and the Brauer-Siegel theorem. The first result about equidistribution of Galois orbits of $\CM$-points on modular curves was established by Duke \cite{duke}.
\begin{proof}[Strategy of the proof]
Let $\mu _S$ be the hyperbolic measure on $S(\C)$ and $\mu_A$ be the normalised Haar measure on $A(\C)$. Assume that $S,A, \Psi$ are defined over a number field $K$ and that $\Gamma$ is contained in the division hull of $A(K)$. We have three main facts preventing the existence of infinitely many points in $\Psi (\CM)\cap \Gamma$:
\begin{itemize}
\item Let $(x_n)_n$ be an infinite sequence of $\CM$-points in $S(\Qbar)$, then the uniform probability measure on the $\Gal(\overline{K}/K)$-orbit of $x_n$ weakly converges, as $n \to \infty$, to the measure $\mu_S$;
\item Let $(a_n)_n$ be an infinite sequence of almost division points relative to $K$, i.e. 
\begin{displaymath}
\lim _{n \to \infty}\sup_{\sigma \in \Gal (\overline{K}/K)} || a_n^\sigma -a_n||=0,
\end{displaymath}
such that $[K(a_n):K]\to \infty$. Then the uniform probability measure on the $\Gal(\overline{K}/K)$-orbit of $a_n$ weakly converges to the measure $\mu_A$;
\item Some measure theoretic lemmas (\cite[Lemma 3.1, 3.2, 3.3]{buiumpoonen1}) preventing this.
\end{itemize}
\end{proof}
\begin{rmk}
Such proof allows also to fatten $\Gamma$: Let $\epsilon > 0$, we may replace $\Gamma$ by $\Gamma_\epsilon := \Gamma + A_\epsilon$, where $A_\epsilon$ is a set of points of small N\'{e}ron-Tate height, i.e. $A_\epsilon := \{ a\in A(\Qbar) \text{  such that  }h(a)\leq \epsilon \}$. The theorem then asserts that, for some $\epsilon > 0$, the set $\Psi (\CM) \cap \Gamma_\epsilon$ is finite, see \cite[Theorem 2.3]{buiumpoonen1}. See also \cite{plus}, where B. Poonen strengthen the Mordell-Lang conjecture by fattening $\Gamma$ in this way.
\end{rmk}

In the subsequent work, \cite{buiumpoonen2}, the authors, motivated by some local results involving the theory of arithmetic differential equations, conjectured that the same results hold when $\CM$-points are replaced by isogeny classes. Recall that a non-cuspidal point $x\in X_1(N)$, defined over $\Qbar$, corresponds to an elliptic curve $E_x /\Qbar$ (with some extra structure), and its isogeny class is defined as the subset of $X_1(N)(\Qbar)$ given by the elliptic curves admitting a $\Qbar$-isogeny to $E_x$ (see \ref{isogeny} for more about the definition). The following is \cite[Conjecture 1.7]{buiumpoonen2}:

\begin{thm}\label{conj}
Let $A$ be an elliptic curve defined over $\Qbar$ and $\Gamma \leq A(\Qbar)$ be a finite rank subgroup. Let $x\in X_1(N)(\Qbar)$ be a non-cuspidal point and $\Sigma_x$ be its isogeny class. Let $X \subset X_1(N)\times A$ be a irreducible closed $\Qbar$-subvariety such that $X(\Qbar) \cap (\Sigma_x \times \Gamma )$ is Zariski dense in $X$, then $X$ is one of the following: $\{\text{point}\}\times \{\text{point}\}$, $X_1(N)\times \{ \text{point}\}$, $\{\text{point}\}\times A$, $X_1(N)\times A$.
\end{thm}
The aim of this paper is to prove Theorem \ref{conj}, as a special case of the following more general result: 
\begin{thm}\label{01}
Let $A, \Gamma, x, \Sigma_x$ be as in Theorem \ref{conj}. If $X \subset X_1(N)\times A$ is an irreducible closed $\Qbar$-subvariety such that $X(\Qbar) \cap (\Sigma_x \times \Gamma_\epsilon)$ is Zariski dense in $X$ for every $\epsilon >0$, then $X$ is one the following: $\{\text{point}\}\times \{\text{point}\}$, $X_1(N)\times \{ \text{point}\}$, $\{\text{point}\}\times A$, $X_1(N)\times A$. 
\end{thm}
\begin{rmk}
In particular taking $X$ in Theorem \ref{01} to be the graph of a non-constant $\Qbar$-morphism $\Psi : X_1(N)\to A$, we get a result analogous to Theorem \ref{thmcm}. Namely we have that, for some $\epsilon >0$, the image of $\Sigma_x$ along $\Psi$ meets $\Gamma_\epsilon$ in only finitely many points.
\end{rmk}

Theorem \ref{conj} may be thought as an Andr\'{e}-Pink-Mordell-Lang conjecture, as will be discussed in section \ref{conjectures}. See also \cite{orrisogenous} for more about the Andr\'{e}-Pink conjecture. It is worth noticing that such conjecture will appear in our result in the form of \cite[Theorem B]{MR3576114} and \cite[Theorem 7.6.]{pinkconjectures}.

Our approach follows the strategy of Buium-Poonen presented above, using a equidistribution result about Hecke points in place of Zhang's equidistribution of $\CM$-points on modular curves and Serre's open image theorem for elliptic curves without complex multiplication. The equidistribution result follows from the work of Clozel, Eskin, Oh and Ullmo, and it is described in section \ref{equid}. Notice that it holds for arbitrary Shimura varieties, in particular it is possible to obtain a result analogue to Theorem \ref{01} for the isogeny class of Galois generic points in higher dimensional Shimura varieties. We remark that, even if the Andr\'{e}-Oort conjecture for $\mathcal{A}_g$ is now a theorem, the equidistribution conjecture for Galois orbits of $\CM$-points is still unsolved. 

\subsection{Related work}
Another fruitful approach for problems like the ones discussed in this paper is to use O-minimality and the Pila-Zannier strategy. This approach relies on the Pila-Wilkie counting theorem, and it was used to prove both the Manin-Mumford and the Andr\'{e}-Oort conjecture. For example Z. Gao, in \cite{gao}, obtained important results towards what he calls the Andr\'{e}-Pink-Zannier conjecture. After the paper was written, it was pointed out to the author that Gabriel Dill employed such strategy to obtain results about unlikely intersections between isogeny orbits and curves similar to the ones presented here. Dill's progress on a modification of the Andr\'{e}-Pink-Zannier conjecture (\cite[Conjecture 1.1]{dill}) implies indeed Theorem \ref{aomm} as well as Theorem \ref{conj} (see \cite[Corollary 1.4]{dill}).

As explained above, our strategy does not rely on O-minimality at any point. The proofs obtained here are quite short but are confined to isogeny classes of Galois generic points. An advantage of the \emph{equidistributional approach} is that it allows to fatten $\Gamma$, by adding points of small N\'{e}ron-Tate height (as in Theorem \ref{01}). Finally notice that our result does not invoke Masser-W\"{u}stholz Isogeny Theorem, as it often happens in results regarding isogeny classes (see \ref{mwsection} for a more detailed discussion about this).

\subsection{Organisation of paper}
In the first section we discuss Hecke orbits on modular curves, and formulate the equidistribution result of Clozel, Eskin, Oh and Ullmo in the form needed for the main theorem. In the second section we prove the conjecture (Corollary \ref{cor}). We end the section showing how to obtain an analogous statement for quaternion Shimura curves $X^D(\mathcal{U})$ (Theorem \ref{quaternion}) and discuss in more details the equidistribution of Hecke points on higher dimensional Shimura varieties. In the last section we present some conjectures about unlikely intersection for a product of a Shimura variety and an abelian variety, inspired by the theorems presented so far. We prove that they follow from the Zilber-Pink conjecture (Proposition \ref{proponconj}).

\subsection*{Acknowledgements}
We thank Andrei Yafaev for regular and inspiring discussions and Martin Orr for valuable comments about the Zilber-Pink conjecture. We also thank Fabrizio Barroero and Gabriel Dill for helpful discussions about the Modified Andr\'{e}-Pink-Zannier conjecture. Finally we are grateful to an anonymous referee for carefully reading this paper. This work was supported by the Engineering and Physical Sciences Research Council [EP/L015234/1], the EPSRC Centre for Doctoral Training in Geometry and Number Theory (The London School of Geometry and Number Theory), University College London.

\section{Notations}
\begin{itemize}
\item By $X\subset Y$ we mean that $X$ injects into $Y$, if we want to say that such injection can not be an isomorphism we write $A \subsetneq B $;
\item We denote by $\A_f$ the (topological) group of finite $\Q$-adeles, i.e. $\A_f = \widehat{\Z}\otimes \Q$, endowed with the adelic topology;
\item As in \cite[Notation 0.2]{deligneshimura} we write $^0$ as an exponent to denote an algebraic connected component and $^+$ for a topological connected component, e.g. $G(\R)^+$ is the topological connected component of the identity of the group of the real points of $G$. We write $G(\R)_+$ for the subgroup of $G(\R)$ of elements that are mapped into the connected component $\Gad(\R)^+\subset \Gad(\R)$, where $\Gad$ denotes the adjoint group of $G$. Finally we set $G(\Q)^+:=G(\Q)\cap G(\R)^+$ and $G(\Q)_+ := G(\Q) \cap G(\R)_+$.
\end{itemize}

\section{Preliminaries}
Let $\Lambda$ be a neat congruence subgroup of $\Sl_2(\Z)$, and $X^+= \Hh$ be the upper half plane, coming with the action of $\Sl_2(\Z)$ by fractional linear transformations. For the purpose of the paper, we may assume $\Lambda$ to be one of $\Gamma (N), \Gamma_0(N), \Gamma_1(N)$ (and $N>3$). A (non-compact) modular curve is a Riemann surface of the form
\begin{displaymath}
S_\Lambda := \Lambda \backslash X^+.
\end{displaymath}
Since in this paper we are interested in maps from modular curves to elliptic curves (which are compact), it is natural to identify the above quotients as the non-cuspidal locus in their Alexandroff compactifications. The compactifications obtained from the choices of $\Lambda$ mentioned above are denoted by $X(N), X_0(N), X_1(N)$. They can be written as an opportune quotient of
\begin{displaymath}
\mathbb{H}^*:= \mathbb{H}\cup \Pp^1(\Q),
\end{displaymath}
and we denote by $\infty_{S_\Lambda}$ the projection of $\infty \in \Pp^1(\Q)\subset \mathbb{H}^*$ onto the compactification of $S_{\Lambda}$.

Points on such complex curves naturally correspond to complex elliptic curves (with some $\Lambda$-structure). Using the moduli interpretation one can show that modular curves are naturally defined over a number field. Let $K\subset \C$ be a field such that $S_\Lambda$ is defined over $K$ and $S_\Lambda(K)\neq \emptyset$. A $K$-point of a modular curve corresponds to an elliptic curve defined over $K$. We usually write
\begin{displaymath}
x \in S_\Lambda(K) \rightsquigarrow E_x /K .
\end{displaymath}
\subsection{Hecke operators}	\label{heckeoperators} For every $a\in \Sl_2(\Q)$, consider the diagram of (Shimura) coverings
\begin{displaymath}
\Lambda \backslash X^+ \xleftarrow{\text{pr}}  (\Lambda \cap a^{-1}\Lambda a) \backslash X^+ \xrightarrow{a \cdot} \Lambda \backslash X^+ .
\end{displaymath}
It induces a finite correspondence, called \emph{Hecke operator}
\begin{displaymath}
T_a : S_\Lambda \longrightarrow S_\Lambda.
\end{displaymath}
The Hecke operator $T_a$ maps a point $x\in S_\Lambda$ to the finite set
\begin{displaymath}
\{ a\lambda x \ |\ \lambda \in (\Lambda \cap a^{-1}\Lambda a) \backslash \Lambda \}.
\end{displaymath}

\begin{rmk}
Since $\Lambda$ is neat, for all $a,b \in \Sl_2(\Q)$ the following are equivalent:
\begin{itemize}
\item $T_{a}(x)\cap T_b(x)\neq \emptyset$ ;
\item $T_{a}(x)= T_b(x)$;
\item $\Lambda a \Lambda = \Lambda b \Lambda$.
\end{itemize}
\end{rmk}

We set 
\begin{displaymath}
\deg _ \Lambda (a):= |  \Lambda a \Lambda / \Lambda| = [\Lambda : a^{-1}\Lambda a \cap \Lambda].
\end{displaymath}
Given a point $x \in S_{\Lambda}$ we denote by $T(x)$ its \emph{Hecke orbit}: 
\begin{displaymath}
T(x):= \bigcup _{g \in \Sl_2(\Q)} T_g (x) \subset S_\Lambda .
\end{displaymath}

The Hecke operator $T_a$ also acts on functions $f$ on $S_\Lambda$ by
\begin{displaymath}
T_af(x) := \frac{1}{\deg _ \Lambda (a)} \cdot \sum _{s \in T_a(x)} f (s).
\end{displaymath}

\begin{exe}
Let $x\in X_0(N)$, corresponding to a pair $(E_x, \Psi_x)$, where $E_x$ is an elliptic curve and $\Psi_x$ is a $\Gamma_0(N)$-level structure (i.e. subgroup of order $N$). Given a prime $p$, not dividing $N$, the Hecke operator $T_p$ applied to $x$ gives
\begin{displaymath}
T_p(x)=T_p(E_x, \Psi_x) = \bigcup_C (E_x/C , (\Psi_x + C )/ C) 
\end{displaymath}
where the union is over all the subgroups $C \subset E_x$ of cardinality $p$.
\end{exe}

\subsubsection{Hecke orbits and isogeny classes}\label{modualrcurves}
Let $S=X_1(N)$, and $x\in S(\Qbar)$ be a non-cuspidal point. It may be represented as a pair $(E_x,P_x)$, where $E_x$ is an elliptic curve defined over $\Qbar$ and $P_x$ is a point of order $N$. We set
\begin{equation}\label{isogeny}
\Sigma^{X_1(N)}_x :=\{(E,P) \text{  such that there exists an isogeny between  } E \text{  and  } E_x \} \subset  S(\Qbar).
\end{equation}
Otherwise stated, we are looking at elliptic curves isogenous to $E$ and the isogeny is not required to respect the points of order $N$. This is the notion of isogeny class appearing in Theorem \ref{conj}.

Consider $\mathcal{A}_1$ the modular curve parametrizing elliptic curves. It is easy to see that $T(x)=\Sigma^{\mathcal{A}_1}_x$ for any $x\in \mathcal{A}_1 (\Qbar)$. By forgetting the point of order $N$, there is a finite Shimura morphism associated to the same Shimura datum
\begin{displaymath}
\pi : X_1(N)\longrightarrow \mathcal{A}_1.
\end{displaymath}
In particular the preimage of an $\mathcal{A}_1$-Hecke orbit can be written as a finite union of Hecke orbits in $X_1(N)$, in symbols
\begin{equation}
\Sigma_x^{X_1(N)}= \bigcup_{i=1}^m T(x_i).
\end{equation} 
This will be the main step in the deduction of Corollary \ref{cor} from Theorem \ref{mainthm}. 

\subsubsection{Hecke orbits and strictly Galois generic points}\label{galoisaction}
Let $K$ be a number field and $x$ a non-cuspidal $K$-point in $S_\Lambda$. To such $x$ there is a corresponding Galois representation
\begin{displaymath}
\rho_{x}: \Gal (\overline{K}/K)\to \overline{\Lambda}\subset \Gl_2(\A_f),
\end{displaymath}
where $\overline{\Lambda}$ denotes the closure of $\Lambda$ in $\Gl_2(\A_f)$. In terms of the associated elliptic curve $E_x$, $\rho_x$ is nothing but the representation coming from the inverse limit of the Galois modules $E_x[n]$. Recall \cite[Definition 6.3]{pinkconjectures}:
\begin{itemize}
\item $x$ is called \emph{Hodge generic/non-special} if the elliptic curve $E_x$ is not $\CM$;
\item $x$ is called \emph{Galois generic} if $\im (\rho_x)$ is open in $\overline{\Lambda}$;
\item $x$ is called \emph{strictly Galois generic} if $\im (\rho_x)$ is equal to  $\overline{\Lambda}$.
\end{itemize}

Let $x\in X_1(N)$ be a non-special non-cuspidal point defined over a number field. Serre's open image theorem asserts that $\im (\rho_x)$ is open in $\overline{\Lambda}$, i.e. that $x$ is Galois generic. See \cite{MR0387283} and \cite{MR0263823} for Serre's proof. 
\begin{rmk}
We remark here that the difference between Galois generic and strictly Galois generic points is not important in this paper. Indeed let $x\in S_\Lambda$ a Galois generic point, we may shrink $\Lambda$ in such a way that $x$ lifts along $\pi : S_{\Lambda'}\to S_\Lambda$ and becomes strictly Galois generic in $S_{\Lambda'}$.
\end{rmk}

It is easy to see that the Hecke orbit $T_a(x)$ of a strictly Galois generic point $x$ is permuted transitively by $\Gal (\overline{K}/K)$. See for example \cite[Proposition 6.6]{pinkconjectures}. In particular
\begin{displaymath}
\forall a\in \Sl_2(\Q) \text{  and  } s\in T_{a}(x), \ \ \ \deg_{\Lambda}(a)=[K(s):K] .
\end{displaymath}

\subsection{Equidistribution of Hecke points}\label{equid}
Let $S=S_\Lambda$ be a modular curve, and write $\mu _S$ for the hyperbolic measure on $S(\C)$. If we fix coordinates $(x,y)$ in $\mathbb{H}$, the measure $\mu_S$ is the measure whose pullback to $\mathbb{H}$ equals a multiple of the hyperbolic measure $y^{-2}\text{d}x\text{d}y$. Given a point $p$ in $S$ we denote by $\delta_p$ the Dirac distribution at $p$.

This is the main result of the section.
\begin{thm}\label{th}
Let $x\in S$ be a strictly Galois generic point defined over a number field $K$. Let $(a_n)_n\subset\Sl_2(\Q)$ be an arbitrary sequence, and fix $s_n \in T_{a_n}(x)$ for every $n$. We have that
\begin{displaymath}
T_{a_n}(x)= \Gal (\overline{K}/ K)s_n.
\end{displaymath}
Moreover, if the cardinality of $\{s_n\}_n$ is not finite then $[K(s_n):K]\to + \infty$ and the sequence of measures
\begin{displaymath}
\Delta_{T_{a_n}(x)}: =\frac{1}{ | \Gal (\overline{K}/K)s_n|} \sum_{p\in \Gal (\overline{K}/K)s_n } \delta_p 
\end{displaymath}
weakly converges to $\mu_S$ as $n\to + \infty$.
\end{thm}
\begin{proof}
As explained in section \ref{galoisaction}, since $x$ is strictly Galois generic point, the Hecke orbit coincides with the Galois orbit. In particular we have
\begin{displaymath}
\deg_\Lambda (a_n)= [K(s_n):K],
\end{displaymath}
and an equality of measures 
\begin{displaymath}
\frac{1}{\deg_\Lambda (a_n)} \sum _{ \lambda \in (\Lambda \cap a_n^{-1}\Lambda a_n) \backslash \Lambda } \delta_\lambda = \frac{1}{ | \Gal (\overline{K}/K)s_n|} \sum_{p\in \Gal (\overline{K}/K)s_n } \delta_p .
\end{displaymath}
From the former equation we see that $\deg_\Lambda (a_n)$ goes to infinity if and only if $[K(s_n):K]$ does. The results of \cite{MR1827734}, together with the existence of infinitely distinct $s_n$, imply that the $\deg_\Lambda (a_n)\to + \infty$ and the desired weakly convergence of measures (see Theorems \ref{tt} and \ref{ttt} for the general statements we are referring to)\footnote{As the reader may have noticed, Masser-W\"{u}stholz Isogeny Theorem shows that the existence of infinitely many distinct $(s_n)$ forces the degree to grow (as explained in section \ref{mwsection}). However we are showing this by invoking an equidistribution results which does not rely on the Isogeny Theorem and holds for arbitrary Shimura varieties.}.
\end{proof}

\section{Main results}
Throughout this section we fix an elliptic curve $A$ defined over $\Qbar$ and a subgroup $\Gamma \leq A(\Qbar)$ of finite rank. Let 
\begin{displaymath}
h: A(\Qbar)\longrightarrow \R_{\geq 0}
\end{displaymath}
be a canonical height function attached to some symmetric ample line bundle on $A$, and, for every $\epsilon \geq 0$, let
\begin{displaymath}
\Gamma_{\epsilon}:= \{ \gamma + a \text{  such that  } \gamma \in \Gamma, \ a \in A(\Qbar), \ h(a)\leq \epsilon \}.
\end{displaymath}
By $S=S_\Lambda$ we will denote a modular curve as in the previous section.
\begin{thm}\label{mainthm}
Let $x$ be Galois generic point of $S$ and $(a_n)_n$ be an arbitrary sequence in $\Sl_2(\Q)$. Let $X \subset S\times A$ be an irreducible closed $\Qbar$-subvariety which is not of the form $S\times \{ \text{point}\}$, $\{\text{point}\}\times A$, $S\times A$. For some $\epsilon>0$, $X(\Qbar)$ contains only finitely many points lying in $\left( \bigcup _n T_{a_n}(x)\right) \times \Gamma_\epsilon$. 
\end{thm}

\begin{rmk}
In particular, by listing all the elements of $\Sl_2(\Q)$, we have also that
\begin{displaymath}
X(\Qbar) \cap \left(  T(x) \times  \Gamma_\epsilon \right)
\end{displaymath}
is finite, where $T(x)$ is defined as $\bigcup _{g \in \Sl_2(\Q)} T_g (x)$.
\end{rmk}

Regarding isogeny classes, as in section \ref{modualrcurves}, and in the direction of Theorem \ref{01} we obtain the following:
\begin{cor}\label{cor}
Suppose $S$ is the modular curve $X_1(N)$ over $\Qbar$. Let $A$, $\Gamma$, $X$ be as in Theorem \ref{mainthm}, and $x$ be a non-cuspidal $\Qbar$-point of $X_1(N)$. For some $\epsilon >0$, $X(\Qbar)$ contains only finitely many points lying in $\Sigma^{X_1(N)}_x  \times \Gamma_\epsilon$.
\end{cor}
\begin{proof}[Proof of Corollary \ref{cor}] 
Let $E_x$ be the elliptic curve (with some extra structure) corresponding to $x$. Recall that a non-cuspidal point in a modular curve is either special or Hodge generic. In terms of the endomorphisms ring of $E_x$ this means that $\End (E_x)\otimes \Q$ is either a quadratic imaginary field, or the field of rational numbers. In the former case the corollary follows from Theorem \ref{thmcm}. Indeed elliptic curves isogenous to a $\CM$-elliptic curve are again $\CM$, therefore the set $\Sigma^{X_1(N)}_x$ is contained in the set of special points of $S$, so
\begin{displaymath}
X(\Qbar) \cap \left( \Sigma^{X_1(N)}_x \times \Gamma_\epsilon \right) \subset X(\Qbar) \cap \left(\CM\times \Gamma_\epsilon \right).
\end{displaymath}
Theorem \ref{thmcm}, in the more general form of \cite[Theorem 2.3]{buiumpoonen1}, shows precisely that the right hand side is finite (for some $\epsilon >0$).

Suppose now that $x$ is Hodge generic. Serre's open image theorem implies that $x$ is Galois generic. The result then follows from Theorem \ref{mainthm} since $\Sigma^{X_1(N)}_x$ is a finite union of Hecke orbits (as explained in section \ref{heckeoperators}, in particular \ref{modualrcurves}).
\end{proof}

\subsection{Proof of Theorem \ref{mainthm}}
In the statement of Theorem \ref{mainthm} $x$ is assumed to be Galois generic. As explained in the remark of section \ref{galoisaction}, we may and do assume that $x$ is \emph{strictly Galois generic}. Indeed there exist $\Lambda '$ and $x_{\Lambda' }\in S_{\Lambda '}$ such that 
\begin{displaymath}
\pi : S_{\Lambda '}\longrightarrow S
\end{displaymath}
maps $x_{\Lambda' }$ to $x$ and $x_{\Lambda' }$ is strictly Galois generic. Since $\pi$ is a finite map, we may replace $X$ by an $X' \subset  S_{\Lambda '} \times A$ which projects onto $X \subset S\times A$ and the validity of the result does not change.

Denote by $\mu _S$ the hyperbolic measure on $S(\C)$ and by $\mu_A$ the normalised Haar measure on $A(\C)$. Define $B_r$ to be the open disk in $S(\C)$ with center $\infty_S$ and radius $r$ with respect to the metric. Lemma 3.3 in \cite{buiumpoonen1} shows that $\mu_S$ blows up relative to the Riemannian metric near the cusp $\infty_S$. Using also \cite[Lemma 3.1]{buiumpoonen1} we can choose a compact annulus $C \subset B_r - \{\infty_S\}$ such that
\begin{equation}\label{eqqq}
\mu _S (C) > \mu_A (\Psi (B_r - \{\infty_S\})),
\end{equation}
for more details see also \cite[pp.6, last but second paragraph]{buiumpoonen1}. From now on we fix such a $C$.

Finally we say that a sequence $(y_n)_n$ in a scheme $X$ is \emph{generic} if it converges to the generic point of $X$ with respect to the Zariski topology, i.e. each proper subvariety of $X$ contains at most finitely many $y_n$.
\begin{proof}[Proof of Theorem \ref{mainthm}]
Of course if the set $\bigcup _n T_{a_n}(x)$ is finite the theorem trivially holds true. Therefore we may and do assume that the set
\begin{displaymath}
\Sigma_x^{(a_n)} :=\bigcup _n T_{a_n}(x)
\end{displaymath}
is infinite. 

Heading for a contradiction let us suppose that $X(\Qbar) \cap ( \Sigma^{(a_n)}_x \times \Gamma_\epsilon ) $ is Zariski dense in $X$ for every $\epsilon >0$. Since $X$ has only countably many subvarieties, we may choose a generic sequence of points $y_n = (s_n,\gamma_n) \in X(\Qbar)$ with $s_n \in \Sigma^{(a_n)}_x$ and $\gamma_n\in \Gamma_{\epsilon_n}$ where $\epsilon_n \to 0$. In particular, each $s_n$ appears only finitely often.

Up to enlarging the base field, we may assume that $A,S,X,x$ are all defined over a number field $K$ and that $\Gamma$ is contained the division hull of $A(K)$. Theorem \ref{th} implies that $[K(s_n):K]\to + \infty$. Since, by assumption, $X$ surjects onto $S$ and $A$ and $X\neq S\times A$ we have that the projection $X\to A$ is generically finite, say of degree $d$. Since $[K(s_n):K]\leq d [K(\gamma_n):K]$, we have also that $[K(\gamma_n):K]\to +\infty$ (as $n$ goes to infinity). The $\gamma_n$s form a sequence of almost division points relative to $K$ in the sense of \cite{almostdivision} and, by passing to a subsequence, we may assume that they admit a coherent limit. Moreover, as $\dim A =1$, the only possibility for the coherent limit of the $\gamma_n$s is $(A,\{0\})$.

The combination of the next two facts implies the contradiction we were aiming for:
\begin{itemize}
\item As explained in Theorem \ref{th}, the uniform probability measure associated to the points $\Gal (\overline{K}/K)s_n$ weakly converges to $\mu_S$ on $S(\C)$, i.e.
\begin{displaymath}
\left ( \frac{1}{ | \Gal (\overline{K}/K)s_n|} \sum_{p\in \Gal (\overline{K}/K)s_n } \delta_p \right)  \longrightarrow \mu_S, \text{  as  } n \to + \infty;
\end{displaymath}
\item Zhang's result, \cite[Theorem 1.1]{almostdivision}, implies that the uniform probability measure on $\Gal (\overline{K}/K)\gamma_n$, as $n\to + \infty$, weakly converges to the Haar measure $\mu_A$ on $A(\C)$.
\end{itemize}
In particular, arguing as in \cite[pp.7, first paragraph]{buiumpoonen1}, they imply that
\begin{displaymath}
\mu_S(C)\leq \mu_A(\Psi (B_r - \{\infty_S\})),
\end{displaymath}
contradicting the choice of $C$ in \ref{eqqq}. The theorem is proven.
\end{proof}

In the rest of the section we discuss the case of more general Shimura varieties.
\subsection{Shimura varieties}
Let $G$ be an almost $\Q$-simple group, $(G,X^+)$ be a connected Shimura datum and $\Lambda$ an arithmetic subgroup of $G(\Q)_+$. In this section we present the general setting for arbitrary connected Shimura varieties 
\begin{displaymath}
S_\Lambda := \Lambda \backslash X^+.
\end{displaymath}

\begin{rmk}
The definitions of the first section naturally generalise to arbitrary Shimura data $(G,X^+)$. Notice that there is a more general notion of \emph{generalized Hecke orbit} which takes into account non-inner automorphisms of $(G,X)$, see \cite[Definition 3.1]{pinkconjectures}. This generalisation does not substantially change the content of the paper. Indeed, when the group $G$ is of adjoint type, the quotient 
\begin{displaymath}
\Aut (G,X^+) / G(\Q)^+
\end{displaymath}
is finite.
\end{rmk}

The main theorem about equidistribution of Hecke points (after Clozel, Eskin, Oh and Ullmo) is the following:
\begin{thm}\label{tt}
Let $(a_n)_n\subset G(\Q)_+$ be an arbitrary sequence of points and $x \in \Lambda \backslash X^+$. Exactly one of the following happens:
\begin{enumerate}
\item The set $\bigcup _n T_{a_n}(x)$ is finite and $\deg_\Lambda( a_n)$ is bounded;
\item The set $\bigcup _n T_{a_n}(x)$ is Zariski dense in $S$ and the sequence of measures $\Delta_{T_{a_n}(x)}$ weakly converges to the canonical Haar measure on $S$. Where we set
\begin{displaymath}
\Delta_{T_{a_n}(x)} : = 	\frac{1}{\deg_\Lambda (a_n)} \sum _{ \lambda \in (\Lambda \cap a^{-1}\Lambda a) \backslash \Lambda } \delta_\lambda
\end{displaymath}
and $\delta_\lambda $ denotes the Dirac distribution at $\lambda$.
\end{enumerate}
\end{thm}
\begin{proof}
See \cite[Corollary 7.2.3]{MR3576114}, which follows from \cite[Theorem 7.2.2]{MR3576114}. In \cite[Section 9.1]{MR3576114} it is also explained how the result can be deduced from \cite{eskinoh} (using \cite[Proposition 2.1]{eskinoh}). When $\Lambda=\GSp2g (\Z)$, see also \cite[Theorem 7.5]{pinkconjectures}, which builds on \cite{MR1827734}.
\end{proof}
It easily implies the next theorem. For a recollection of facts about Galois and Hodge generic points on arbitrary Shimura varieties the reader may also consult \cite[Section 1.4]{galoisgeneric}.
\begin{thm}\label{ttt}
Let $S_\Lambda $ be a connected Shimura variety and $x\in S$ be a strictly Galois generic point defined over a number field $K$. Let $(a_n)_n \subset G(\Q)_+$ be an arbitrary sequence, and fix $s_n \in T_{a_n}(x)$ for every $n$. We have that
\begin{displaymath}
T_{a_n}(x)= \Gal (\overline{K}/ K)s_n.
\end{displaymath}
Moreover, if the cardinality of $\{s_n\}_n$ is not finite then $[K(s_n):K]\to + \infty$ and the sequence of measures
\begin{displaymath}
\Delta_{T_{a_n}(x)}=\frac{1}{ | \Gal (\overline{K}/K)s_n|} \sum_{p\in \Gal (\overline{K}/K)s_n } \delta_p 
\end{displaymath}
weakly converges to the hyperbolic measure $\mu_S$ on $S(\C)$ as $n\to + \infty$.
\end{thm}

\subsection{Quaternion Shimura curves}
Let  $D$ be a non-split indefinite quaternion algebra over $\Q$ and fix a maximal order $\Oo_D$. In this section we prove a statement analogous to Corollary \ref{cor} for quaternion Shimura curves $X^D(\mathcal{U})/\Qbar$, i.e. the Shimura curve attached to $(D,\mathcal{U})$, where $\mathcal{U}$ is a sufficiently small compact subgroup of $(\Oo_D\otimes \widehat{\Z})^*$ such that $X^D(\mathcal{U})$ is connected. See \cite{kevin} for a complete treatment about such Shimura curves. Such curves parametrise \emph{fake elliptic curve} (with a $\mathcal{U}$-level structure), i.e. abelian surfaces $E$ with an embedding $\Oo_D \subset \End (E)$. Using such interpretation we have a notion of isogeny class $\Sigma_x \in X^D(\mathcal{U})$ as in section \ref{modualrcurves}.

We have a version of Serre's open image theorem which holds for arbitrary Shimura curves:
\begin{thm}
Let $S_\Lambda$ be a Shimura curve. A $\Qbar$-point $x\in S$ is either special or Galois generic.
\end{thm}
\begin{proof}
On a Shimura curve a point is either special or Hodge generic. The main theorem of \cite{MR0419368} shows precisely that Hodge generic points are Galois generic (the proof is similar to the methods used by Serre). Since $X^D(\mathcal{U})$ parametrises fake elliptic curves, it is possible to prove this special case directly from Serre's open image, without invoking Ohta's theorem.
\end{proof} 
Such open image theorem, combined with Theorem \ref{ttt} implies the equidistribution of the Hecke orbit associated to a Hodge generic point, as used in the case of $X_1(N)$. The equidistribution of the Galois orbit of $\CM$-points, as used in \cite[Theorem 2.5 and Theorem 2.6]{buiumpoonen1}, follows again from Brauer-Siegel and Zhang's paper \cite{MR2200081}. 

To obtain a contradiction in this case, it is enough to use \cite[Lemma 3.6]{buiumpoonen1}. Indeed let $\Psi : S \to A$ be a map from a Shimura curve\footnote{In the previous theorem we had to use a different strategy since the compact Riemann surface $X_1(N)$ is \emph{only} the compactification of a Shimura curve.} to an elliptic curve, \cite[Lemma 3.6]{buiumpoonen1} shows that $\Psi_* \mu_S \neq \mu _A$. Therefore we cannot have a sequence of measures $\Delta_{s_n}$ weakly converging to $\mu_S$, whose pushforward, $\Psi^* \Delta_{s_n}$, weakly converges to $\mu_A$.

We have eventually proved the following:
\begin{thm}\label{quaternion}
Let $S / \Qbar$ a quaternion Shimura curve. Let $A$, $\Gamma$, $\Psi$ be as in Theorem \ref{mainthm}. Let $x$ be a $\Qbar$-point of $S$. There exists an $\epsilon >0$, such that image of a isogeny class $\Sigma_x \subset S(\Qbar)$ along $\Psi$ intersects $\Gamma_\epsilon$ in only finitely many points.
\end{thm}

\subsection{A remark on Masser-W\"{u}stholz Isogeny Theorem}\label{mwsection}
This is the main theorem of \cite{mw} (see also \cite{MR3263028} for a bound that does not depend on the polarisations).
\begin{thm}[Masser-W\"{u}stholz]\label{mw}
Let $A, B$ be principally polarised abelian varieties of dimension $g$ over a number field $K$, and suppose that $A_\C$ and $B_\C$ are isogenous. Then if we let $N$ be the minimal degree of an isogeny between them over $\C$, we have 
\begin{displaymath}
N \leq b_g \max (h_{Fal}(A), [K:\Q])^{c_g},
\end{displaymath}
where $b_g, c_g$ are positive constants depending only on $g$ and $h_{Fal}(A)$ denotes the semistable Faltings height of $A$.
\end{thm}

It has the following amusing consequence regarding the field of definition of the Hecke points. For the proof see \cite[Lemma 9.2.1]{MR3576114}.
\begin{cor}\label{cormw}
Let $(G,X^+)$ be a connected Shimura datum of abelian type. Let $x\in S_\Lambda$ be a point with residue field $K$. For every integer $d$ there are only finitely many $t\in T(x)$ such that the degree of $K(t)$ over $K$ is bounded by $d$.
\end{cor}
It is interesting to notice that Theorem \ref{mw} is used in the proof of partial results towards conjectures about unlikely intersections. For example in the AO and Andr\'{e}-Pink conjectures. Corollary \ref{cormw} may be used for the result of this paper. Indeed, for Shimura varieties of abelian type, it implies the existence of finitely many Hecke operators of bounded degree. In our approach we deduced this from Theorem \ref{tt}, which builds on different techniques. 
\begin{rmk}
For example Corollary \ref{cormw} may be applied to arbitrary Shimura curves. This is possible since all Shimura curves are of abelian type, as proven by Deligne in \cite[Section 6]{delignetravaux}.
\end{rmk}
\section{Mixed Shimura varieties and the Zilber-Pink conjecture}\label{conjectures}
In the introduction we discussed characterisations of subvarieties of a product of a modular curve and an elliptic curve intersecting a dense set of \emph{special} points (Theorems \ref{aomm}, \ref{thmcm}, \ref{conj}). We formulate analogous conjectures for products of higher dimensional Shimura varieties and abelian varieties. In section \ref{zproof}, we prove that they follow from the Zilber-Pink conjecture about unlikely intersections in \emph{mixed} Shimura varieties (Conjecture \ref{zpconj}).

Throughout this section let $T:=S \times A$ be the product of $S$ a Shimura variety and $A$ an abelian variety (of dimension $g$). When we do not specify the field of definition of an object, we assume that it is defined over the field of complex numbers. 

Recall that we have notions of being special and weakly special for both subvarieties of Shimura and abelian varieties, in particular we denote by $\CM \subset S(\Qbar)$ the subset of special points of $S$. In this section we combine the two as follows. For an overview about special subvarieties and the Andr\'{e}-Oort conjecture, we refer the reader to \cite{MR3821177}.
\begin{defi}
A \emph{special} (resp. \emph{weakly special}) subvariety of $T$ is a subvariety of the form $S'\times A'$ where $S'$ is a special (resp. weakly special) subvariety of $S$ and $A'$ is a special (resp. weakly special) subvariety of $A$. We say that a subvariety of $T$ is \emph{weakly special generic} if it not contained in any smaller weakly special subvariety of $T$.
\end{defi}

We state three conjectures about a weakly special generic closed irreducible subvariety $X \varsubsetneq T$.
\begin{conj}[Andr\'{e}-Oort-Manin-Mumford]\label{aommgeneral}
The subset of special points of $T$ is not Zariski dense in $X$.
\end{conj}

We denote by $A^{[> d]}$ the union of all algebraic subgroups of $A$ of codimension $>d$ and we fix a subgroup of finite rank $\Gamma \leq A(\C)$.
\begin{conj}[Andr\'{e}-Oort-Mordell-Lang]\label{aomlgeneral}
The set $(\CM\times (A^{[>\dim X]} + \Gamma) ) \cap X$ is not Zariski dense in $X$.
\end{conj}

In the next conjecture, by isogeny class of a point $s\in S \subset \mathcal{A}_{g'}$, corresponding to an abelian variety $A_s$, we mean the set of points $s'\in S$ corresponding to abelian varieties $A_{s'}$ isogenous to $A_s$.
\begin{conj}[Andr\'{e}-Pink-Mordell-Lang]\label{apmlgeneral}
Assume $S$ is a sub-Shimura variety of $\mathcal{A}_{g'}$ for some $g'$, and let $\Sigma_s$ be the isogeny class of a point $s\in S(\C)$. The set $(\Sigma_s \times \Gamma ) \cap X$ is not Zariski dense in $X$.
\end{conj}

When the abelian variety $A$ is defined over $\Qbar$, we can also fatten $\Gamma$, by replacing $\Gamma$ by $\Gamma_\epsilon$, to formulate an Andr\'{e}-Oort-Mordell-Lang-Bogomolov Conjecture and an Andr\'{e}-Pink-Mordell-Lang-Bogomolov Conjecture. Theorem \ref{01} is a special case of such formulation, requiring all the objects to be defined over $\Qbar$. Since the aim of the section is a comparison with the conjectures appearing in the work of Pink (\cite{pinkconjectures}, \cite{pinkprepreint}), we discuss only the case of subgroups of finite rank.

\begin{rmk}
Combining the recent proof of the AO conjecture for Shimura varieties of abelian type (culminated in \cite{ts}) and the proof of Manin-Mumford (\cite{MR2411018}), it is possible to prove Conjecture \ref{aommgeneral} whenever $S$ is a Shimura variety of abelian type.
\end{rmk}

For recent developments, using O-minimality, towards the Andr\'{e}-Pink-Mordell-Lang we point out to the reader the main theorems of G.\ Dill, see \cite{dill}. See also the main theorems of \cite{gao}. Indeed, as mentioned in the introduction, Conjecture \ref{apmlgeneral} formally follows from Gao's Andr\'{e}-Pink-Zannier\footnote{Only a small modification is needed, in order to take into account non-polarised isogenies and subgroups of arbitrary finite rank.}.

\subsection{Zilber-Pink conjecture}\label{zp}
To state the Zilber-Pink conjecture (\cite[Conjecture 1.1]{pinkprepreint}) we need to introduce some vocabulary from the theory of mixed Shimura varieties. For a complete treatment we refer the reader to \cite[Section 2]{pinkconjectures}, \cite[Chapter VI]{MR1044823} and \cite{MR1128753}.

\subsubsection{Mixed Shimura varieties}
Let $\DT:= \Res_{\C / \R}\Gm$ be the Deligne torus. A connected mixed Shimura datum is a pair $(P,X^+)$ where
\begin{itemize}
\item  $P$ is a connected linear algebraic group defined over $\Q$, with unipotent radical $W$, and an algebraic subgroup $U \subset W$ which is normal in $P$;
\item $X^+\subset \Hom (\DT_{\C}, P_\C)$ is a connected component of an orbit under the subgroup $P(\R)\cdot U(\C)\subset P(\C)$;
\end{itemize}
satisfying axioms (i)-(vi) in \cite[Definition 2.1]{pinkconjectures}. A connected mixed Shimura variety associated to $(P,X^+)$ is a complex manifold of the form $\Lambda \backslash X^+$ where $\Lambda$ is a congruence subgroup of $P(\Q)_+$ acting freely on $X^+$.

A mixed Shimura datum allows to take into consideration groups of the form $\GSp2g \ltimes \Ga^{2g}$. For suitable congruence subgroups the associated connected mixed Shimura variety is the universal family of abelian varieties over the moduli space of principally polarised abelian varieties (with some $n$-level structure). The point is that every (principally polarised) abelian variety can be realised as a fibre of such a family. 

As for the pure case, there is a notion of special and weakly special subvarieties of mixed Shimura varieties (see \cite[Section 4]{pinkconjectures}). As the reader may expect every irreducible component of the intersection of special subvarieties (resp. weakly special) is again special (resp. weakly special), and a weakly special subvariety containing a special point is itself special. For example special points in the universal family of abelian varieties correspond to torsion points in the fibers $A_s$ over all special points $s\in \Ag$.

Finally pure Shimura varieties are also mixed Shimura varieties (they occur precisely when $P$ is reductive) and a product of (finitely many) mixed Shimura varieties is again a mixed Shimura variety.
\subsubsection{The conjecture}\label{zproof}
\begin{conj}[Zilber-Pink]\label{zpconj}
Consider a mixed Shimura variety $M$ over $\C$ and a Hodge generic irreducible closed subvariety $X \subset M$. Then the intersection of $X$ with the union of all special subvarieties of $M$ of codimension $> \dim X$ is not Zariski dense in $X$.
\end{conj}

The proof of the next proposition is similar to the arguments appearing in Theorem 3.3, 5.3 and 5.7 of the preprint \cite{pinkprepreint}. See also \cite[Section 8]{gao} and \cite{pinkconjectures}, where it is explained that Andr\'{e}-Pink(-Zannier) for mixed Shimura varieties implies Mordell-Lang (\cite[Theorem 5.4]{pinkconjectures}).
\begin{prop}\label{proponconj}
Conjecture \ref{zpconj} implies both Conjecture \ref{aomlgeneral} and Conjecture \ref{apmlgeneral}.
\end{prop}
We first fix some notations. Let $a\in \mathcal{A}_{g,n}$ be the point corresponding to the abelian variety $A$ (for some $n \geq 3$), $S^a$ the smallest Shimura subvariety of $ \mathcal{A}_{g,n}$ containing $a$ and $M$ the universal abelian scheme over $S^a$. The variety $S\times M$ is a mixed Shimura variety and it contains $S\times A = S\times M_a$, where $M_a$ denotes the fibre of
\begin{displaymath}
\pi : M \to S^a
\end{displaymath}
over the point $a$.

Finally fix a maximal sequence of linearly independent elements $a_1, \dots, a_n \in \Gamma$, and let $C$ the Zariski closure of the subgroup of $A^n$ generated by the point $\underline{a}:=(a_1,\dots, a_n)$. We may assume $C$ is an abelian variety. Moreover, since $a$ is Hodge generic in $S^a$, we may view $C$ as the fibre over $a$ of an $S^a$ flat subgroup scheme $Z$ of the $n$-th fibred power of $M$ (cf. the discussion at the beginning of the proof of \cite[Theorem 5.7]{pinkprepreint}). We will apply the Zilber-Pink conjecture to the subvarieties of the mixed Shimura variety
\begin{displaymath}
B:=S \times \left( M \times_{S^a}  Z \right).
\end{displaymath}  
\begin{proof}[Zilber-Pink implies Andr\'{e}-Oort-Mordell-Lang]
Equivalently we may suppose that $X$ is not contained in any special subvariety of $T$ and deduce that $(\CM\times \Gamma ) \cap X$ is not Zariski dense in $X$. Consider the irreducible closed subvariety of $B$ defined by 
\begin{displaymath}
Y:= X\times \{\underline{a}\}.
\end{displaymath}
Since $a$ is Hodge generic in $S^a$, $X$ is weakly special generic in $T$, $\underline{a}$ is Zariski dense in $C$, then $Y$ is a Hodge generic subvariety of $B$ of dimension $\dim X$.

As for abelian varieties, we denote by $ M^{[>d]}$ the union of $M_x ^{[>d]}$, varying $x $ in $S^a$. To conclude, applying the Zilber-Pink conjecture to $Y\subset B$, we only need to show that the set
\begin{displaymath}
( X \cap (\CM \times ( M^{[>\dim X]}+\Gamma) ) )  \times \{\underline{a}\}
\end{displaymath}
is contained in the intersection between $Y$ and the union of all special subvarieties of $B$ of codimension $> \dim X = \dim Y$. Let $G$ be a $S^a$-flat algebraic subgroup of $M$ of codimension $> \dim X$, and let $x=(c,g+ \gamma) \in X$, where $c$ is a special point in $S$ and $\gamma$ an element of $\Gamma$. For some integer $m>0$ we may write
\begin{displaymath}
m \gamma = m_1 a_1 + \dots +  m_n a_n,
\end{displaymath}
then we have $m\gamma = \varphi (\underline{a})$ for the homomorphism of $S^a$-group schemes $\varphi := (m_1,\dots, m_n): M^n \to M$. We may therefore write $(x,\underline{a})$ as an element in the set
\begin{displaymath}
H:=\CM \times (m^{-1}(G\times \{0\}+ (\varphi,m)(C))).
\end{displaymath}
Since the codimension of $H$ in $B$ is bigger than $\dim X$, we have proved the desired inclusion.

The result on the fiber over $a$ then follows in virtue of the following remark (see proof of \cite[Theorem 5.7]{pinkprepreint}): Let $X\subset A$ be an irreducible closed subvariety of $A$, $X$ is contained in a proper algebraic subgroup of $A$ if and only it is contained in a special subvariety of $M$ of codimension $>0$.
\end{proof}

\begin{proof}[Zilber-Pink implies Andr\'{e}-Pink-Mordell-Lang]
By applying Hecke operators, we may assume that $X\times \{\underline{a}\}$ and $\{s\}\times\{\underline{a}\}$ lie in a given connected component of $B$. Let $S^s$ be the smallest Shimura subvariety containing $s$, and $S'$ the smallest Shimura subvariety of $B \times \mathcal{A}_{g'}$ containing $Y\times \{s\}$.

Suppose $( X \cap (\Sigma_s \times ( M^{[>\dim X]}+\Gamma) ) ) $ is not Zariski dense in $X$, we want to prove that $X$ is weakly special, more precisely we show that $X\times \{s\}$ is an irreducible component of a fibre of $S '\to S^s$. To do so we apply Zilber-Pink (actually in the equivalent form appearing in \cite[Conjecture 1.1]{pinkprepreint}) to
\begin{displaymath}
\Sigma_s \times \{s\} \times ( M^{[>\dim X]}+\Gamma) \times \{\underline{a}\}.
\end{displaymath}
The result follows by combining the argument presented in the previous proof, which allows to see the points in $\Gamma$ as special points in an opportune Shimura variety (see also the last paragraph in the proof of \cite[Theorem 5.3]{pinkprepreint}), and the argument of \cite[Theorem 3.3]{pinkprepreint} (noticing that given two points $s,t\in \mathcal{A}_{g'}$ such that the underlying abelian varieties are isogenous, then the defect of $s\in \mathcal{A}_{g'}$ is equal to the defect of $(s,t)\in \mathcal{A}_{g'}^2$, as in \cite[Lemma 2.2]{orrisogenous}). 
\end{proof}

\bibliographystyle{alpha}
\bibliography{biblio.bib}

\Addresses

\end{document}